\documentclass[12pt,reqno]{amsart}
\usepackage{dsfont}
\usepackage{amssymb, amsmath, amsthm}
\usepackage{enumerate,color}
\newtheorem{thm}{Theorem}[section]
\newtheorem{lem}[thm]{Lemma}

\newtheorem{prop}[thm]{Proposition}
\newtheorem{defn}[thm]{Definition}
\newtheorem{rmk}{Remark}

\numberwithin{equation}{section}

\newcommand{\bel}{\begin{equation} \label}
\newcommand{\ee}{\end{equation}}
\def\beq{\begin{equation}}
\def\eeq{\end{equation}}
\newcommand{\bea}{\begin{eqnarray}}
\newcommand{\eea}{\end{eqnarray}}
\newcommand{\beas}{\begin{eqnarray*}}
\newcommand{\eeas}{\end{eqnarray*}}
\newcommand{\pd}{\partial}

\newcommand{\ppp}{\partial}

\newcommand{\re}{\mathfrak R}
\newcommand{\im}{\mathfrak I}
\newcommand{\R}{\mathbb{R}}
\newcommand{\C}{\mathbb{C}} 
\newcommand{\N}{\mathbb{N}}

\newcommand{\cA}{\mathcal{A}}
\newcommand{\cB}{\mathcal{B}}
\newcommand{\cC}{\mathcal{C}}

\newcommand{\cE}{\mathcal{E}}
\newcommand{\cL}{\mathcal{L}}
\newcommand{\cO}{\mathcal{O}}
\newcommand{\cS}{\mathcal{S}}
\newcommand{\Hin}{\mathcal{H}_{\mathrm{in}}}
\newcommand{\Gi}{S_{\mathrm{in}}}
\newcommand{\Go}{S_{\mathrm{out}}}

\allowdisplaybreaks

\def\epsilon{\varepsilon}
\def\phi {\varphi}

\def\ppp{\partial}

\providecommand{\norm}[1]{\left\lVert#1\right\rVert}

\renewcommand{\leq}{\leqslant}
\renewcommand{\geq}{\geqslant}

\providecommand{\norm}[1]{\left\lVert#1\right\rVert}
\def\thefootnote{{}}

\title{\bf On time-fractional diffusion equations with space-dependent 
variable order}
\author{
Yavar Kian$^*$,
Eric Soccorsi$^*$,
and Masahiro Yamamoto$^\dag$
}
\date{}

\begin{document}

\begin{abstract}
We investigate diffusion equations with time-fractional derivatives of 
space-dependent variable order. We establish the well-posedness issue and 
prove that the space-dependent variable order coefficient is uniquely 
determined by the knowledge of 
a suitable time-sequence of partial Dirichlet-to-Neumann maps.
\end{abstract}

\maketitle

\renewcommand{\thefootnote}{\fnsymbol{footnote}}
\footnotetext{\hspace*{-5mm} 
\begin{tabular}{@{}r@{}p{13cm}@{}} 
\\
$^\dag$& Graduate School of Mathematical Sciences, 
the University of Tokyo, 
3-8-1 Komaba, Meguro-ku, Tokyo 153-8914, Japan. 
E-mail:
myama@next.odn.ne.jp
\\
$^*$& Aix Marseille Universit\'e, Universit\'e de Toulon, CNRS, CPT, 
Marseille, France.
\end{tabular}}

\section{Introduction}

\subsection{Statement of the problem}

Let $\Omega$ be a bounded domain of $\R^d$, $d \geq 2$, with Lipschitz 
continuous boundary $\partial \Omega$, and let $(a_{i,j})_{1 \leq i,j \leq d} 
\in L^\infty(\Omega;\R^{d^2})$ be symmetric,
i.e., fulfill $a_{i,j}=a_{j,i}$ a.e. in $\Omega$, for $i,j = 1,\ldots,d$,
and satisfy the ellipticity condition
\bel{ell}
\exists c>0,\ \sum_{i,j=1}^d a_{i,j}(x) \xi_i \xi_j \geq c |\xi|^2,\ x \in 
\Omega,\ \xi=(\xi_1,\ldots,\xi_d) \in \R^d.
\ee
For $\kappa \in (d,+\infty]$ and $q \in L^\kappa(\Omega)$, such that
\bel{a9}
q(x) \ge 0,\ x \in \Omega, 
\ee
we introduce the formal differential operators 
$$ 
\cA_0 u(x)=-\sum_{i,j=1}^d \partial_{x_i}\left(a_{i,j}(x) \partial_{x_j} u(x) 
\right)\ \mbox{and}\ \cA_q u(x) := \cA_0 u(x) + q(x) u(x),\  x \in\Omega, 
$$
where we set $\ppp_{x_i} = \frac{\ppp}{\ppp x_i}$, $i=1,\ldots,d$.

Given $T\in(0,+\infty]$ and two functions $\alpha \in L^\infty(\Omega)$ and 
$\rho \in L^\infty(\Omega)$ satisyfing 
\bel{eq-rho}
0 < \alpha_0 \leq \alpha(x) \leq \alpha_M<1\ \mbox{and}\ 0<\rho_0 \leq\rho(x) 
\leq\rho_M <+\infty,\ x \in \Omega, 
\ee
we consider the initial-boundary value problem for 
a space-dependent variable order (VO) fractional diffusion equation
\bel{eq1}
\left\{ \begin{array}{rcll} 
(\rho(x)\partial_t^{\alpha(x)} +\cA_q) u(t,x) & = & f(t,x), & (t,x)\in  
Q:=(0,T) \times \Omega,\\
u(t,x) & = & 0, & (t,x) \in \Sigma := (0,T) \times \pd \Omega, \\  
u(0,x) & = & u_0(x), & x \in \Omega.
\end{array}
\right.
\ee
Here and below, $\partial_t^{\alpha(x)}$ denotes the Caputo fractional 
derivative of order $\alpha(x)$ with respect to $t$, defined by 
$$
\partial_t^{\alpha(x)} u(t,x):=\frac{1}{\Gamma(1-\alpha(x))}\int_0^t
(t-s)^{-\alpha(x)} \partial_s u(s,x) d s,\ (t,x) \in Q,
$$
where $\Gamma$ is the Gamma function.

In this paper, we pursue two goals. The first one is to establish the 
well-posedness of the initial-boundary value problem \eqref{eq1} 
for a suitable source term $f$ and initial value $u_0$. 
The second one is to analyse the uniqueness in an inverse problem 
of determining simultaneously the fractional order $\alpha$ and 
two coefficients $\rho$ and $q$ of the diffusion equation in 
\eqref{eq1} by partial Cauchy data.

\subsection{Physical motivations}

Anomalous diffusion in complex media is a rapidly growing field of academic 
research with multiple engineering applications in geophysics, environmental 
science and biology. The diffusion properties of homogeneous porous media are 
currently modeled, see e.g., \cite{AG,CSLG}, by constant order (CO) 
time-fractional diffusion processes where in \eqref{eq1} the mapping 
$x \mapsto \alpha(x)$ is maintained constant over $\Omega$. 
However, in complex media, the presence of heterogeneous regions causes 
variations of the permeability in different spatial positions, and in this 
case, the VO time-fractional model is more relevant for describing the 
diffusion process, and see e.g., \cite{SCC}. This is a background for 
VO time-fractional diffusion equations.


\subsection{A short review of the mathematical literature of time-fractional 
diffusion equations}

Ordinary and partial differential equations with fractional derivatives have 
been intensively studied  over the two last decades. We refer to 
\cite{KST,MR,P, SM} for a general introduction to fractional calculus, and 
for example to \cite{A, GM,L} for more specific foci on partial differential 
equations with time fractional derivatives. The well-posedness problem for CO 
time-fractional diffusion equations was addressed in \cite{BY, GLY, SY, Z1} 
and see also \cite{DAV}, where the local existence, the uniqueness and 
the continuity on initial data of the solution to partial integrodifferential 
equations of parabolic type are discussed.
The time decay of their solutions was studied in \cite{KSVZ}. Recently, a new 
definition of the weak solution to these equations was introduced in \cite{KY},
which allows for defining solutions to semi-linear fractional wave equations. 
Moreover, initial-boundary value problems for multi-terms time-fractional 
diffusion equations were studied by \cite{BY, LLY}. 
Notice that a De Giorgi-Nash H\"older regularity theorem was derived in 
\cite{ACV} (see also \cite{Z2}) for solutions to CO time-fractional equations
with fractional diffusion in space. 
As for distributed order (DO) time-fractional diffusion equations, we refer 
to \cite{LKS} for the analysis of the well-posedness problem, and to 
\cite{LKS,LLY2} for the study of the asymptotic behavior of the solution. 
However, in contrast with CO or DO time-fractional equations, 
for VO time-fractional diffusion equations. to our best 
knowledge, there are no results available in the mathematical literature.

Quite similarly, there is only a small number of mathematical papers dealing 
with inverse problems associated with time-fractional diffusion processes, 
which are listed below. In the one-dimensional case, \cite{CNYY} proved 
simultaneous determination of the constant fractional differential order and 
the time-independent diffusion coefficient by Dirichlet boundary measurements 
for the solution. In dimensions 2 or larger, \cite{HNWY} determined 
a constant fractional order from measurements at one point of the solution 
over the entire time span. 
In \cite{FK, SY}, the time-varying factor in the source term or in the zeroth 
order coefficient of time-fractional equations was stably determined by 
pointwise observation of the solution. For half-order fractional diffusion 
equations, \cite{CXY,YZ} proved stability in determining a zeroth order 
coefficient by means of a Carleman estimate. An inverse boundary value 
problem for diffusion equations with multiple fractional time derivatives is 
examined in \cite{LIY} and the authors prove the uniqueness in determining the 
number of fractional time-derivative terms, the orders of the derivatives, 
and spatially varying coefficients. Finally, in \cite{KOSY}, the zeroth and 
first order space-dependent coefficients defined on a Riemanian manifold, 
along with the Riemanian metric, are simultaneously determined by a 
partial Dirichlet-to-Neumann map taken at one arbitrarily fixed time.

\subsection{Main results}

The first result of this paper is given for a Lipschitz continuous bounded 
domain $\Omega$. It establishes the existence, the uniqueness and the 
regularity properties of the weak solution to the initial-boundary 
value problem \eqref{eq1} in the sense of Definition \ref{d1} below.
For the statement of the mail result, we introduce the contour in $\C$:
\bel{g1} 
\gamma(\epsilon,\theta) := \gamma_-(\epsilon,\theta) \cup \gamma_0(\epsilon,
\theta) \cup \gamma_+(\epsilon,\theta)
\ee
with $\epsilon \in (0,1)$ and $\theta \in \left( \frac{\pi }{2},\pi \right)$,
where 
\bel{g2}
\gamma_0(\epsilon,\theta) := \{ \epsilon e^{i \beta};\ \beta \in
[-\theta,\theta] \}\ \mbox{and}\ \gamma_\pm(\epsilon,\theta) 
:= \{s e^{\pm i \theta};\ s \in [\epsilon,+\infty) \}
\ee
and the double sign corresponds each other.

Henceforth $\langle t \rangle$ stands for 
$(1+t^2)^{\frac{1}{2}}$, and the interval $(0,T]$ (resp., $[0,T]$) should 
be understood as $(0,+\infty)$ (resp., $[0,+\infty)$) for the case of
$T=+\infty$.
\\

Furthermore by $A_q$ we denote the self-adjoint realization in $L^2(\Omega)$ of
the operator $\cA_q$ with the homogeneous Dirichlet boundary
condition and for $p \in \C \setminus \R_-$, 
by $(A_q+\rho(x)p^{\alpha(x)})^{-1}$ the resolvent operator 
of $A_q+\rho(x)p^{\alpha(x)}$.

Henceforth $\mathcal{B}(X,Y)$ denotes the Banach space of all the bounded 
linear operators from a Banach space $X$ to another Banach space $Y$, 
and we set 
$\mathcal{B}(X) = \mathcal{B}(X,X)$.

Then the existence and uniqueness result of a weak solution to \eqref{eq1} 
is as follows.

\begin{thm}
\label{t1} 
Suppose that \eqref{ell} and \eqref{a9} are fulfilled. 
Let $u_0\in L^2(\Omega)$.  
We assume that $f \in L^\infty(0,T;L^2(\Omega))\cap 
\mathcal C((0,T];L^2(\Omega))$ in the case of $T < +\infty$, and 
$f \in  \mathcal C((0,+\infty);L^2(\Omega))$ satisfies 
$\langle t \rangle^{-\zeta} f \in L^\infty(\R_+;L^2(\Omega))$ with 
some $\zeta \in \R_+$ in the case of $T=+\infty$.
Then there exists a unique weak solution $u \in\mathcal C((0,T];L^2(\Omega))$ 
to \eqref{eq1}, which is expressed by
\bel{t1a} 
u(t) = u(t,\cdot) = S_0(t) u_0 + \int_0^t S_1(t-\tau) f(\tau) d \tau+ S_2 f(t),
\ t \in (0,T],
\ee
where we set
$$ 
S_0(t) \psi := \frac{1}{2 i \pi} \int_{\gamma(\epsilon,\theta)} 
e^{t p} (A_q+\rho(x)p^{\alpha(x)})^{-1} \rho(x)p^{\alpha(x)-1} \psi dp,
$$
$$ 
S_1(t) \psi := \frac{1}{2 i \pi} \int_{\gamma(\epsilon,\theta)} 
e^{t p} (A_q+\rho(x)p^{\alpha(x)})^{-1} \psi dp
$$
and
$$ 
S_2 \psi := \frac{1}{2 i \pi} \int_{\gamma(\epsilon,\theta)} 
p^{-1}(A_q+\rho(x)p^{\alpha(x)})^{-1} \psi dp
$$
for all $\psi \in L^2(\Omega)$, 
the three above integrals being independent of the choice of 
$\epsilon \in (0,1)$ and $\theta \in \left( \frac{\pi }{2},\pi \right)$.

Moreover, if $f=0$, then the mapping $u: (0,T) \longrightarrow 
L^2(\Omega)$ is analytic in $(0,T)$.
\end{thm}

\begin{rmk}
\label{rm-cl}
We point out that for all $\alpha_0 \in \left( 0, \frac{1}{2} \right)$, 
the operator $S_2$ is identically zero, provided we have $\alpha_M  \in 
(\alpha_0, 2 \alpha_0)$. 
Therefore \eqref{t1a} is reduced to the classical Duhamel formula:
\bel{t1a-cl}
u(t) = S_0(t) u_0 + \int_0^t S_1(t-\tau) f(\tau) d \tau,\ t \in (0,T].
\ee
\end{rmk}

The second result deals with the inverse problem of determining the unknown 
coefficients $\alpha$, $\rho$, $q$ of the time fractional diffusion 
equation in \eqref{eq1} by partial boundary data of the solution.
More precisely, we assume that $\partial \Omega$ is $\cC^{1,1}$ and 
\bel{eq-DLaplace}
a_{i,j}(x)=\delta_{i,j},\ x \in \Omega,\ i,j =1,\ldots,d,
\ee
where $\delta_{i,j}$ is equal to $1$ whenever $i=j$, and to $0$ otherwise.
Then we fix $k \in \N \setminus \{ 0  , 1 \}$ and consider the following 
system
\bel{eq11}
\left\{\begin{array}{rcll} 
(\rho(x)\partial_t^{\alpha(x)} +\cA_q) u(t,x) & = & 0, & (t,x)\in  (0,+\infty)\times\Omega,\\
u(t,x) & = & t^k g(x), & (t,x) \in (0,+\infty)\times\pd\Omega, \\  
u(0,x) & = & 0, & x \in \Omega,
\end{array}
\right.
\ee
with suitable $g$. 
Given two non empty subsets $\Gi$ and $\Go$ of $\ppp\Omega$, 
we introduce the following boundary operator
\bel{bound}
\mathcal N_{\alpha,\rho,q}(t) : \Hin \ni g \mapsto \partial_\nu u_g(t,\cdot)
_{|\Go}, \quad t \in (0,+\infty),
\ee
where $ \Hin := \{g\in H^{3\slash2}(\partial\Omega);\ \textrm{supp}\ g 
\subset \overline{\Gi} \}$. 
Here by $u_g$ we denote a unique solution in $\cC([0,+\infty);H^2(\Omega))$
to \eqref{eq11}, whose existence is guaranteed by Proposition \ref{pr1}
stated below, $\nu$ is the outward normal unit vector to $\pd \Omega$, 
and $\pd_\nu u_g(t,x) := \nabla u_g(t,x) \cdot \nu(x)$ for $(t,x) \in 
(0,+\infty)\times\pd\Omega$.

We discuss the uniqueness in the inverse problem of determining 
the coefficients $(\alpha,\rho,q)$ from the knowledge of the 
boundary operators $\{ \mathcal N_{\alpha,\rho,q}(t_n); \thinspace n \in \N \}$
associated with a time sequence $t_n$, $n \in \N$ fulfilling
\bel{cond1}
\mbox{the set $\{ t_n; \ n \in \N \}$ has an accumulation point in 
$(0,+\infty)$.}
\ee
Moreover we assume that $\Omega$, $\Gi$ and $\Go$ satisfy the 
following conditions.
\begin{enumerate}[(i)]
\item {\bf Case: $d=2$.}\\
It is required that $\partial\Omega$ is composed of a finite number of 
smooth closed contours.  In this case, we choose $\Gi=\Go :=\gamma$, 
where $\gamma$ is any arbitrary non-empty relatively open subset of 
$\partial\Omega$, and the set of admissible unknown coefficients reads
\beas
\cE_2 & := &  \left\{ (\alpha,\rho,q);\ \alpha \in W^{1,r}(\Omega)\ \mbox{and}
\ \rho \in W^{1,r}(\Omega)\ \mbox{fulfill}\ \eqref{eq-rho}\ \mbox{and} \right.
\\
& & \left. \ \ \ q\in W^{1,r}(\Omega;\R_+)\ \mbox{with}\ r \in (2,+\infty) 
\right\}.
\eeas
\item {\bf Case: $d \geq 3$.}\\
We choose $x_0 \in \R^d$ outside the convex hull of $\overline{\Omega}$.
Then we assume that
$$ \{x\in\partial\Omega;\ (x-x_0)\cdot\nu \geq 0\} \subset 
\Gi\ \mbox{and}\ \{x\in\partial\Omega;\ (x-x_0)\cdot\nu\leq0\} \subset \Go.
$$
Furthermore we define the set of admissible unknown coefficients by
$$
\cE_d := \left\{ (\alpha,\rho,q); \ \alpha \in L^\infty(\Omega)\ \mbox{and}\ 
\rho \in L^\infty(\Omega)\ \mbox{fulfill}\ \eqref{eq-rho}\ \mbox{and}\ 
q \in L^\infty(\Omega;\R_+) \right\}. 
$$
\end{enumerate}
The uniqueness result for our inverse coefficients problem is as follows.

\begin{thm}
\label{It2} 
Let $t_n$, $n\in\N$ fulfill \eqref{cond1} and assume 
that either $(i)$ or $(ii)$ is satisfied.
If 
\bel{It2a}
\mathcal N_{\alpha_1,\rho_1,q_1}(t_n)=\mathcal N_{\alpha_2,\rho_2,q_2}(t_n),
\ n\in\N, \quad (\alpha_j,\rho_j,q_j) \in \cE_d, \thinspace j=1,2,  
\ee
then we have $(\alpha_1,\rho_1,q_1) = (\alpha_2,\rho_2,q_2)$.
\end{thm}

\subsection{Comments and outline}
As the Laplace transform of a solution to CO time-fractional diffusion 
equations is expressed in terms of Mittag-Leffler functions, most of its 
features are inherited from the well known properties of these special 
functions.
As will appear below, this is no longer the case when the fractional order of 
the time-fractional diffusion equation depends on the space variable, which 
makes for a more challenging analysis of the well-posedness of these systems. 
This new technical difficulty translates in particular into the definition of 
a weak solution to VO time-fractional diffusion equations, which extends the 
one of a weak solution to CO time-fractional diffusion equations. Moreover, 
it can be verified from \cite[Definition 1.1]{LKS} that such a definition is 
suitable for DO time-fractional diffusion equations as well.

The paper is organized as follows. In Section \ref{sec-fwd}, 
we discuss the well-posedness of the initial-boundary value problem 
\eqref{eq1}. 
More precisely, the weak solution to the VO time-fractional diffusion equation 
appearing in \eqref{eq1}, is defined in Section \ref{sec-wks}, and Section 
\ref{sec-prt1} proves Theorem \ref{t1}, which is by means of 
a technical resolvent estimate of the elliptic part of the diffusion equation
given in Section \ref{sec-ell}. The proof of the statement of Remark 
\ref{rm-cl} can be found in Section \ref{sec-pr-cl}.

The analysis of the uniqueness result in our inverse problem
is carried out in Section \ref{sec-inv}.  That is, the partial boundary 
operators \eqref{bound} are rigorously defined in Section \ref{sec-wp}, and 
Section \ref{sec-ip} provides the proof of Theorem \ref{It2}.


\section{Analysis of the forward problem}
\label{sec-fwd}

\subsection{Elliptic operator: self-adjointness and resolvent estimate}
\label{sec-ell}
Let $A_0$ be the operator generated by the quadratic form
$$ a_0(u) := \sum_{i,j=1}^d \int_{\Omega} a_{i,j}(x) \pd_{x_i} u(x) 
\pd_{x_j} u(x) dx,\ u \in H_0^1(\Omega). $$
Since there exists a constant $\tilde{c}_0>0$ such that
\bel{a6}
a_0(u) \geq c_0 \| \nabla u \|_{L^2(\Omega)^d}^2 
\geq \tilde{c}_0 \| u \|_{H^1(\Omega)}^2,\ u \in H_0^1(\Omega),
\ee
by \eqref{ell} and the Poincar\'e inequality, the operator $A_0$ is 
self-adjoint in $L^2(\Omega)$ and acts as $\cA_0$ on its dense domain $D(A_0)$ 
in $L^2(\Omega)$, that is, $A_0 u = \cA_0 u$ for all $u \in D(A_0)$. 

Put $r:=2 \kappa \slash (\kappa-2)$ and notice from the H\"older inequality 
that 
\bel{a7} 
\| q u \|_{L^2(\Omega)} \leq \| q \|_{L^\kappa(\Omega)} \| u \|_{L^r(\Omega)},
\ u \in L^r(\Omega).
\ee
Furthermore we have $H^1(\Omega)=W^{1,2}(\Omega) \subset W^{r_0,r}(\Omega)$ 
with $r_0 := 1 - d \slash \kappa \in (0,1)$ by the Sobolev embedding theorem 
(e.g., \cite[Theorem 1.4.4.1]{Gr}), and the embedding is continuous:
\bel{a8} 
\exists c>0,\ \| u \|_{W^{r_0,r}(\Omega)} \leq c \| u \|_{H^1(\Omega)},\ 
u \in H^1(\Omega).
\ee
Therefore, by \eqref{a7}-\eqref{a8}, we have $\| q u \|_{L^2(\Omega)} \leq 
c \| q \|_{L^\kappa(\Omega)} \| u \|_{H^1(\Omega)}$ 
for every $u \in H^1(\Omega)$, and so it follows from \eqref{a6} that
\beas
\| q u \|_{L^2(\Omega)}^2 & \leq & \frac{c^2 \| q \|_{L^\kappa(\Omega)}^2}
{\tilde{c}_0^2} \langle A_0 u , u \rangle_{L^2(\Omega)} \\
& \leq & \frac{c^2 \| q \|_{L^\kappa(\Omega)}^2}{2 \tilde{c}_0^2} 
\left( \epsilon \| A_0 u \|_{L^2(\Omega)}^2 + \epsilon^{-1} 
\| u \|_{L^2(\Omega)}^2 \right),\ u \in D(A_0),\ \epsilon \in (0,+\infty).
\eeas
Thus, taking $\epsilon>0$ so small that 
$\epsilon c^2 \| q \|_{L^\kappa(\Omega)}^2 < 2 \tilde{c}_0^2$, we see that 
the multiplier by $q$ in $L^2(\Omega)$ is $A_0$-bounded with relative bound 
zero.  
As a consequence, $A_q:=A_0+q$ is self-adjoint in $L^2(\Omega)$ with domain 
$D(A_q)=D(A_0)$  
by the Kato-Rellich theorem (see e.g., \cite[Theorem V.4.3]{Kato}, 
\cite[Theorem X.12]{RS2}). 
Moreover $A_q$ acts as $\cA_q$ on $D(A_q)=D(A_0)$.

In this article, we suppose \eqref{a9} in such a way that 
$A_q \geq \tilde{c}_0$ in the operator sense, where $\tilde{c}_0$ is 
the constant appearing in \eqref{a6}.
This hypothesis is quite convenient for proving Proposition \ref{l1} below 
stated, which is essential for the proof of Theorem \ref{t1} and 
Proposition \ref{pr1}, but it could be removed at the price of greater 
unessential technical difficulties. Nevertheless, for simplicity, 
we shall not go further into this direction. 

\begin{prop}
\label{l1}
For all $p \in \mathbb C \setminus \R_-$, the operator 
$A_q+\rho(x)p^{\alpha(x)}$ is boundedly invertible in $L^2(\Omega)$ and 
$(A_q+\rho(x)p^{\alpha(x)})^{-1}$ maps $L^2(\Omega)$ into $D(A_0)$. 
Moreover  
\bel{l1a}
\norm{(A_q+ \rho(x)r^{\alpha(x)} e^{i \beta \alpha(x)})^{-1}}
_{\cB(L^2(\Omega))} 
\leq C(r,\beta) \max_{j=0,M} r^{-\alpha_j},\ r \in (0,+\infty), \
\beta \in (-\pi,\pi)
\ee
with
\bel{l1b}
C(r,\beta):=\left\{ \begin{array}{cl} 2 \rho_0^{-1}, & \mbox{if}\ | \beta | 
\leq  \theta_*(r), \\
\rho_0^{-1} c_*(\beta), & \mbox{otherwise}, \end{array} \right.
\ee
and
\bel{l1c}
\theta_*(r):= \alpha_M^{-1} \underset{\sigma=\pm 1}{\min} 
\arctan \left( \frac{\rho_0}{3 \rho_M} r^{\sigma (\alpha_M - \alpha_0)} 
\right),\
c_*(\beta) := \max_{j=0,M} |\sin(\alpha_j \beta)|^{-1}.
\ee
Furthermore the mapping $p \mapsto (A_q+\rho(x)p^{\alpha(x)})^{-1}$ is bounded 
holomorphic in $\mathbb C \setminus \R_-$ as operator with values in 
$\mathcal{B}(L^2(\Omega))$.
\end{prop}
\begin{proof} 
We shall prove \eqref{l1a} only for $r \in [1,+\infty)$, because the 
corresponding estimate for $r \in (0,1)$ can be derived in the same way.

a) Firstly we assume that $\beta \in (0,\pi)$.  The case of 
$\beta \in (-\pi,0)$ is similarly treated. 
We define a multiplication operator $U_\beta$ in $L^2(\Omega)$ by 
$U_{\beta}f(x) = u_{\beta}(x)f(x)$ for $f \in L^2(\Omega)$, where
$$ 
u_\beta(x) := \left( \rho(x) r^{\alpha(x)} \sin(\beta \alpha(x) 
\right)^{1 \slash 2},\ x \in \Omega. 
$$
Then $i U_\beta^2$ is the skew-adjoint part of the operator $A_q+\rho(x) 
r^{\alpha(x)} e^{i \beta \alpha(x)}$.
Putting $m_\beta:=\min_{j=0,M} \sin(\alpha_j \beta)$, we have 
$$ 
0 < \rho_0^{1\slash2} m_\beta^{1 \slash 2} r^{\alpha_0 \slash 2} 
\leq u_\beta(x) \leq \rho_M^{1 \slash 2} r^{\alpha_M \slash 2},\ x \in \Omega.
$$
Hence the self-adjoint operator $U_\beta$ is bounded and boundedly invertible 
in $L^2(\Omega)$ and
\bel{a11}
\| U_\beta^{-1} \|_{\cB(L^2(\Omega))} \leq \rho_0^{-1\slash2} m_\beta^{-1 \slash 2} r^{-\alpha_0 \slash 2}.
\ee
Moreover, for each $p=r e^{i \beta}$, it holds true that
\bel{a12}
A_q + \rho(x)p^{\alpha(x)} = U_\beta \left( U_\beta^{-1} B_{q,p} U_\beta^{-1} + i \right) U_\beta,
\ee
where $B_{q,p}:= A_q + \rho(x)r^{\alpha(x)} \cos(\beta \alpha(x))$ denotes 
the self-adjoint part of the operator $A_q + \rho(x)p^{\alpha(x)}$, i.e.,
$$ 
A_q + \rho(x)p^{\alpha(x)} = B_{q,p} + i U_\beta^2. 
$$
Indeed, the multiplication operator by $\rho(x)r^{\alpha(x)} 
\cos(\beta \alpha(x))$ is bounded by $\rho_M r^{\alpha_M}$ in 
$L^2(\Omega)$, and the operator $B_{q,p}$ is self-adjoint in $L^2(\Omega)$ 
with domain $D(A_q)=D(A_0)$, by the Kato-Rellich theorem. 
Thus $U_\beta^{-1} B_{q,p} U_\beta^{-1}$ is self-adjoint in $L^2(\Omega)$ 
as well with domain $U_\beta D(A_0)$.
Therefore, the operator $U_\beta^{-1} B_{q,p} U_\beta^{-1}+i$ is invertible in 
$L^2(\Omega)$ and satisfies the estimate
$$ 
\| ( U_\beta^{-1} B_{q,p} U_\beta^{-1}+i )^{-1} \|_{\cB(L^2(\Omega))} \leq 1. 
$$
It follows from this and \eqref{a12} that $A_q+\rho(x)p^{\alpha(x)}$ is 
invertible in $L^2(\Omega)$ with
$$ 
(A_q+\rho(x)p^{\alpha(x)})^{-1} = U_\beta^{-1} (U_\beta^{-1} B_{q,p} 
U_\beta^{-1} +i )^{-1} U_\beta^{-1}, 
$$
showing that $(A_q+\rho(x)p^{\alpha(x)})^{-1}$ maps $L^2(\Omega)$ into 
$U_\beta^{-1} D(U_\beta^{-1} B_{q,p} U_\beta^{-1})=D(A_0)$.  As a consequence, 
we infer from \eqref{a11} that
$$
\| (A_q+ \rho(x) p^{\alpha(x)})^{-1} \|_{\cB(L^2(\Omega))} 
\leq \| (U_\beta^{-1} B_{q,p} U_\beta^{-1}+i )^{-1} \|_{\cB(L^2(\Omega))} 
\| U_\beta^{-1} \|_{\cB(L^2(\Omega))}^2 
\leq \rho_0^{-1} m_\beta^{-1} r^{-\alpha_0}.
$$

b) It follows from $r \in [1,+ \infty)$ that $\alpha_M \theta_*(r) 
= \arctan ( \rho_0 \rho_M^{-1} r^{-(\alpha_M-\alpha_0)} \slash 3) \in 
(0, \pi \slash 6)$. Thus we have
$\cos( \alpha_M \theta_*(r)) \slash 3 = \rho_0^{-1} 
\rho_M r^{\alpha_M-\alpha_0} \sin(\alpha_M \theta_*(r))$ and
\\ 
$2 \cos( \alpha_M \theta_*(r)) \slash 3 > 1 \slash \sqrt{3} > 1 \slash 2$, 
which entails
\bel{a12b} 
\cos (\alpha_M \theta_*(r)) \geq  \frac{1}{2} + \rho_0^{-1} 
\rho_M r^{\alpha_M-\alpha_0} \sin(\alpha_M \theta_*(r)).
\ee

Next, for each $\beta \in [-\theta_*(r),\theta_*(r)]$, we have $\alpha(x) | 
\beta | \leq \alpha_M  \theta_*(r) < \pi \slash 2$, and hence
$\cos(\alpha(x) \beta) \geq \cos(\alpha_M \theta_*(r))>0$. 
From this, \eqref{ell} and \eqref{a9} it follows that
$$ 
B_{q,p} \geq \rho_0 r^{\alpha_0} \cos(\alpha_M \theta_*(r)) 
$$
in the operator sense. Therefore $B_{q,p}$ is boundedly invertible in 
$L^2(\Omega)$ and
\bel{a13}
\| B_{q,p}^{-1} \|_{\cB(L^2(\Omega))} \leq \frac{\rho_0^{-1}
r^{-\alpha_0}}{\cos(\alpha_M \theta_*(r))}.
\ee
Similarly, by $|\sin(\alpha(x) \beta)| \leq \sin(\alpha_M \theta_*(r))$, 
we obtain that 
$$ 
\| U_\beta \|_{\cB(L^2(\Omega))} \leq \rho_M^{1 \slash 2} 
r^{\alpha_M \slash 2} \sin(\alpha_M \theta_*(r))^{1 \slash 2}. 
$$ 
This and \eqref{a13} yield
$$ 
\| B_{q,p}^{-1} U_\beta^2 \|_{\cB(L^2(\Omega))} \leq \| B_{q,p}^{-1} \|
_{\cB(L^2(\Omega))} \| U_\beta \|_{\cB(L^2(\Omega))}^2 \leq \rho_0^{-1} 
\rho_M r^{\alpha_M - \alpha_0} \tan(\alpha_M \theta_*(r))  < 1. 
$$
Furthermore, using that $A_q+\rho(x) p^{\alpha(x)} 
= B_{q,p} ( I + i B_{q,p}^{-1} U_\beta^2 )$, where $I$ denotes the identity 
operator in $L^2(\Omega)$, we see that $A_q+\rho(x) p^{\alpha(x)}$ is 
invertible in $L^2(\Omega)$ with
\beas
 \|  ( A_q+\rho(x)p^{\alpha(x)} )^{-1} \|_{\cB(L^2(\Omega))} 
& \leq&  \frac{\| B_{q,p}^{-1} \|_{\cB(L^2(\Omega))}}{1 - \| B_{q,p}^{-1} 
U_\beta^2 \|_{\cB(L^2(\Omega))}}\\
 & \leq&  \frac{\rho_0^{-1} r^{-\alpha_0}}{\cos(\alpha_M \theta_*(r)) 
- \rho_0^{-1} \rho_M  r^{\alpha_M-\alpha_0} \sin(\alpha_M \theta_*(r))}. 
 \eeas
This and \eqref{a12b} entail that 
$\| ( A_q+\rho(x)p^{\alpha(x)} )^{-1} \|_{\cB(L^2(\Omega))} \leq 2 \rho_0^{-1} 
r^{-\alpha_0}$.

c) We turn now to proving that $p \mapsto (A_q+\rho(x)p^{\alpha(x)})^{-1}$ is 
bounded holomorphic in $\mathbb C \setminus \R_-$.
To this purpose, we introduce the closed sequilinear form 
$$ 
a_{q,p}(u) := a_0(u) + \int_{\Omega} (q(x) +\rho(x)p^{\alpha(x)}) 
| u(x) |^2 dx,\ u \in H_0^1(\Omega), 
$$
which is associated with the operator $A_q+\rho(x) p^{\alpha(x)}$ in 
$L^2(\Omega)$. 
In light of \eqref{a6} and \eqref{a9}, we have
$$ 
\re{a_{q,p}(u)} \geq \left( \tilde{c}_0 - \rho_M \max_{j=0,M} | p |^{\alpha_j} 
\right) \| u \|_{L^2(\Omega)}^2\ \mbox{and}\ \im{a_{q,p}(u)} \leq \left(  
\rho_M \max_{j=0,M} | p |^{\alpha_j} \right) \| u \|_{L^2(\Omega)}^2 
$$
for all $u \in H_0^1(\Omega)$, entailing that $a_{q,p}$ is sectorial 
for every $p \in \mathbb C \setminus \R_-$. 
Here and henceforth $\re$ and $\im$ mean the real part and the imaginary 
part of a complex number under consideration, respectively.

Moreover, since $p \mapsto a_{q,p}(u)$ is holomorphic in $\mathbb C \setminus 
\R_-$, we see that 
$\{ a_{q,p};\thinspace p \in \mathbb C \setminus \R_- \}$ is an analytic 
family of sesquilinear forms of type (a) in the sense of Kato 
(see \cite[Section VII.4.2]{Kato}). 
From \cite[Theorem VII.4.2]{Kato} it follows that 
$\{ A_q + \rho(x) p^{\alpha(x)};\ p \in \mathbb C \setminus \R_- \}$ is an analytic family of operators. 
Therefore $p \mapsto (A_q + \rho(x) p^{\alpha(x)})^{-1}$ is holomorphic 
in $\mathbb C \setminus \R_-$ by \cite[Theorem VII.1.3]{Kato}.
Thus the proof of Proposition 2.1 is complete.
\end{proof}

We point out that $\theta_*(r)$ behaves likes $\min_{\sigma = \pm 1} 
r^{\sigma(\alpha_M - \alpha_0)}$ as $r$ becomes either sufficiently small or 
sufficiently large (that is, like $r^{-(\alpha_M - \alpha_0)}$ 
as $r \to 0$, and 
like $r^{(\alpha_M - \alpha_0)}$ as $r \to +\infty$). 
Indeed, bearing in mind 
that $\arctan u = \int_0^u \frac{d v}{1+v^2}$ for all $u \in [0,+\infty)$, 
we see that 
$\arctan u \in \left[ \frac{u}{1+u^2}, u \right]$, and so
we infer from \eqref{l1c} that
$$
\frac{\frac{\rho_0}{3  \rho_M} \min_{\sigma = \pm 1} r^{\sigma(\alpha_
M - \alpha_0)}}{1 + \frac{\rho_0^2}{9  \rho_M^2} \min_{\sigma = \pm 1} 
r^{2 \sigma(\alpha_M - \alpha_0)}} \leq \alpha_M \theta_*(r) 
\leq \frac{\rho_0}{3  \rho_M} \min_{\sigma = \pm 1} 
r^{\sigma(\alpha_M - \alpha_0)},\ r \in (0,+\infty).
$$
Since $\min_{\sigma = \pm 1} r^{\sigma(\alpha_M - \alpha_0)} \in (0,1]$, 
the denominator of the left-hand side of the above inequality is majorized 
by $1 + \frac{\rho_0^2}{9  \rho_M^2}  \leq \frac{10}{9}$, so that we have
$$
\frac{3 \rho_0}{10  \rho_M} \min_{\sigma = \pm 1} 
r^{\sigma(\alpha_M - \alpha_0)} \leq \alpha_M \theta_*(r) \leq \frac{\rho_0}
{3  \rho_M} \min_{\sigma = \pm 1} r^{\sigma(\alpha_M - \alpha_0)},\ r 
\in (0,+\infty).
$$
Therefore it follows readily from \eqref{l1b}-\eqref{l1c} and the inequality 
$\sin u \geq \frac{u}{2}$ for all $u \in [0,1]$ that
$$
C(r,\beta) \leq \rho_0^{-1} \left( \sin \left (\frac{3 \alpha_0 \rho_0}
{10  \alpha_M \rho_M} \min_{\sigma=\pm 1} r^{\sigma(\alpha_M-\alpha_0)} 
\right) \right)^{-1} 
\leq \frac{20 \alpha_M \rho_M}{3 \alpha_0 \rho_0^2} \max_{\sigma=\pm 1} 
r^{\sigma(\alpha_M-\alpha_0)},
$$
when $r$ is sufficiently close to either $0$ or $+\infty$. As a consequence, 
there exists a constant $C>0$, which is independent of $r$ and $\beta$ 
such that we have
\bel{a14}
C(r,\beta) \leq C \max_{\sigma=\pm 1} r^{\sigma(\alpha_M-\alpha_0)},\ r \in 
(0,+\infty),\ \beta \in (-\pi,\pi). 
\ee

\subsection{Weak solution}
\label{sec-wks}
Let $\cS'(\R;L^2(\Omega))$ be the space dual to $\cS(\R;L^2(\Omega))$. 
We denote by
$\cS'(\R_+;L^2(\Omega)):=\{v \in \cS'(\R;L^2(\Omega));\ \mbox{supp}
\ v \subset [0,+\infty) \times \overline{\Omega}\}$ the set of distributions 
in $\cS'(\R;L^2(\Omega))$ supported in $[0,+\infty) \times \overline{\Omega}$.
Otherwise stated, $v \in \cS'(\R;L^2(\Omega))$ lies in $\cS'(\R_+;L^2(\Omega))$
if and only if
$\langle v , \varphi \rangle_{\cS'(\R;L^2(\Omega)),\cS(\R;L^2(\Omega))} =0$
whenever $\varphi \in \cS(\R;L^2(\Omega))$ 
vanishes in $\R_+ \times \overline{\Omega}$. As a consequence, for a.e. 
$x\in\Omega$, we have
\bel{a1}
\langle v(\cdot,x) , \varphi \rangle_{\cS'(\R),\cS(\R)}=\langle v(\cdot,x), 
\psi \rangle_{\cS'(\R),\cS(\R)},\ \ \varphi , \psi \in \cS(\R),
\ee
provided $\varphi=\psi$ in $\R_+ $. 
Furthermore we say that $\varphi \in \cS(\R_+)$ if $\varphi$ is the 
restriction to $\R_+$ of a function $\tilde{\varphi}\in \cS(\R)$. 
Then we set
\bel{a2} 
x\mapsto\langle v(\cdot,x) , \varphi \rangle_{\cS'(\R_+),\cS(\R_+)} 
:= x\mapsto\langle v(\cdot,x)  , \tilde{\varphi} \rangle_{\cS'(\R),\cS(\R)},
\ v \in \cS'(\R_+;L^2(\Omega)).
\ee
Notice from \eqref{a1} that $\tilde{\varphi}$ may be any function in $\cS(\R)$ 
such that $\tilde{\varphi}(t)=\varphi(t)$ for all $t \in \R_+$.

For $p \in \mathbb C_+:=\{z\in\mathbb C;\ \re z>0\}$, we put
$$ 
e_p(t) := \exp (-p t ),\ t \in \R_+. 
$$
Evidently, $e_p \in \cS(\R_+)$. For $v \in \cS'(\R_+;L^2(\Omega))$, 
we define the Laplace transform  $\cL[v]$ in $t$ of $v$ by
$$
\cL [v](p) := x\mapsto\langle v(\cdot,x) , e_p \rangle_{\cS'(\R_+),\cS(\R_+)},
\ p \in \mathbb C_+,
$$ 
and notice that $p \mapsto \cL [v](p) \in \cC^\infty(\mathbb C_+;L^2(\Omega))$.
Having seen this, we define the weak solution to \eqref{eq1} as follows.

\begin{defn}
\label{d1} 
Let $u_0 \in L^2(\Omega)$. For $T<+\infty$, we assume that 
$f \in L^1(0,T;L^2(\Omega))$ and, for $T=+\infty$, we assume that there exists 
$m\in\mathbb N$ such that $(1+t)^{-m}f\in L^1(\R_+;L^2(\Omega))$. We say that 
$u$ is a weak solution to \eqref{eq1} if $u$ is the restriction to $Q$ of 
a distribution $v\in \cS'(\R_+;L^2(\Omega))$ and 
the Laplace transform $V:=\cL[v]$ verifies
\bel{a3}
V(p):=\left( A_q +\rho(x)p^{\alpha(x)} \right)^{-1} 
\left(  F(p)+\rho(x)p^{\alpha(x)-1}u_0 \right),\ p \in (0,+\infty).
\ee 

Here $F(p):=\cL[f(t,\cdot)1_{(0,T)}(t)](p)=\int_0^{T}e^{-pt}f(t,\cdot)dt$, 
where $1_{I}$ denotes the characteristic function of a set $I \subset \R$.
\end{defn}

\begin{rmk} 
Notice from \eqref{a3} and Lemma \ref{l1} that $V(p) \in D(A_q) = D(A_0) 
\subset H_0^1(\Omega)$ for all $p \in (0,+\infty)$, which entails that 
$V(p) = 0$ on $\pd \Omega$. 
Actually it is clear that \eqref{a3} can be 
equivalently replaced by the condition
$$
\left\{ \begin{array}{rcll} 
(\cA_q +\rho(x) p^{\alpha(x)}) V(p) & = & F(p)+\rho(x) p^{\alpha(x)-1}u_0, & \mbox{in}\ \Omega, \\  
V(p) & = & 0, & \mbox{on}\ \pd \Omega
\end{array} \right. 
$$
for all $p \in (0,+\infty)$.
\end{rmk}

\begin{rmk} 
For all $h \in \cC^1(\R^+)$ such that
$$ 
\epsilon_0 := \inf \left\{ \epsilon \in (0,+\infty);\ e^{-\epsilon t} 
\frac{d^k h(t)}{dt^k} \in L^1(\R^+),\ k=0,1 \right\} \in \R_+, 
$$
we know from \cite[Eq. (2.140)]{P} that
$$
\cL[\partial^{\alpha(x)} h](p)=p^{\alpha(x)} H(p)-p^{\alpha(x)-1} h(0),
\ p \in (\epsilon_0,+\infty),
$$
where $H(p):=\cL[h](p)=\int_0^{+\infty} e^{-pt}h(t)dt$. Therefore, in the 
particular case where the mapping $x \mapsto \alpha(x)$ is constant, we 
infer from \cite[Theorems 2.1 and 2.2]{SY} that the initial-boundary value 
problem \eqref{eq1} admits a unique weak solution to \eqref{eq1} in the sense 
of Definition \ref{d1}, provided $u_0$ and $f$ are sufficiently smooth.
\end{rmk}

\subsection{Proof of Theorem \ref{t1}}
\label{sec-prt1}
The proof is divided into two parts: the first one is concerned with $S_0$, 
while the second deals with the operators $S_1$ and $S_2$.\\
1) Let us start with $S_0(t)$. To this purpose, we set
$$ 
W(p):=p^{-2}(A_q+\rho(x)p^{\alpha(x)})^{-1} \rho(x)p^{\alpha(x)-1}u_0,\ 
p \in \mathbb \C \setminus \R_-, 
$$
fix $\mu \in [1,+\infty)$, and infer from \eqref{l1a} that
\bel{ee}
\norm{W(\mu+i\eta)}_{L^2(\Omega)}\leq C(r,\beta) 
\rho_M |\mu+i\eta|^{-3+\alpha_M-\alpha_0},\ \eta \in \R,
\ee
where $C(r,\beta)$ is given by \eqref{l1b}-\eqref{l1c} with $r=|\mu+i\eta| 
\in [1,+\infty)$ and $\beta=\arg(\mu+i\eta) \in \left(-\frac{\pi}{2},
\frac{\pi}{2} \right)$. 
According to \eqref{a14}, there exists a constant $C=C(\alpha_0,\alpha_M,
\rho_0,\rho_M)$ such that $$C(r,\beta)\leq C|\mu+i\eta|^{\alpha_M-\alpha_0},\quad \mu\in[1,+\infty),\ \eta\in\R.$$ 
Thus \eqref{ee} yields 
\bel{ess}
\norm{W(\mu+i\eta)}_{L^2(\Omega)} \leq C \langle \eta \rangle^{-3+2(\alpha_
M-\alpha_0)},\ \mu \in [1,+\infty),\ \eta\in\R,
\ee
upon substituting $C$ for $\rho_M C$. 
As a consequence, we have for each $k=1,2$, 
\bel{sup}
C_k:=\sup_{\mu \in [1,+\infty)} \norm{W(\mu+i \cdot)}_{L^k(\R;L^2(\Omega))} 
=  \sup_{\mu \in [1,+\infty)} \left( \int_{\R} \norm{W(\mu+i \eta)}
_{L^2(\Omega)}^k d \eta \right)^{\frac{1}{k}} <\infty,
\ee
and hence
\bel{bm} 
\omega(t):=  \frac{1}{2i\pi} \int_{i\infty}^{i\infty} e^{t p} W(p+1) dp 
= \frac{1}{2 \pi} \int_{-\infty}^{+\infty} e^{ i t\eta} W(1 + i \eta ) d \eta
\ee
is well defined for all $t \in \R$. Moreover the mapping $p \mapsto e^{t p} 
W(p+1)$ is holomorphic in $\mathbb C \setminus (-\infty,-1]$ 
by Proposition \ref{l1}, and so we infer from the Cauchy formula that
\bel{bm1} 
\omega(t)= \frac{1}{2i\pi} \int_{s-i\infty}^{s+i\infty} e^{t p} W(p+1) dp,
\ s\in(0,+\infty).
\ee
Indeed, for all $R \in (1,+\infty)$ and $s \in (0,+\infty)$, we have
\bel{cau}  
\int_{s-iR}^{s+iR} e^{t p} W(p+1) dp -\int_{-iR}^{iR} e^{t p} W(p+1) dp 
= \sum_{\sigma=\pm 1} \sigma \int_0^s e^{t (\mu + i \sigma R)} W(\mu+1+i 
\sigma R) d\mu 
\ee
from the Cauchy formula, and 
$$
\norm{\int_0^s e^{t (\mu + i \sigma R)} W(\mu +1+i \sigma R) d\mu}
_{L^2(\Omega)} \leq C s \max(1,e^{s t}) \langle R \rangle^{-3+2(\alpha_
M-\alpha_0)},\ \sigma = \pm 1,
$$
by \eqref{ess}.  Hence \eqref{bm1} follows by letting $R$ to $+\infty$ 
in \eqref{cau}. 
Next, in view of \eqref{bm1}, we obtain that
$$
\norm{\omega(t)}_{L^2(\Omega)}= \frac{1}{2\pi} \norm{ \int_{\R} 
e^{t (s + i \eta)} W(s+1+i \eta) d \eta}_{L^2(\Omega)}
\leq \frac{e^{st}}{2 \pi} \sup_{\mu \in [1,+\infty)} \norm{W(\mu+i \cdot)}
_{L^1(\R;L^2(\Omega))} 
$$
for all $t \in\R$ and $s \in (0,+\infty)$, and consequently that
\bel{z1}
\norm{\omega(t)}_{L^2(\Omega)} \leq \frac{C_1}{2 \pi} e^{ts}
\ee
according to \eqref{sup}. Now, letting $s$ to $+\infty$ on the right-hand 
side of \eqref{z1}, we have
\bel{ww}
\omega(t)=0,\ t \in(-\infty,0).
\ee
Similarly, by letting $s$ to $0$ in \eqref{z1}, we find that 
$\norm{\omega(t)}_{L^2(\Omega)}  \leq \frac{C_1}{2 \pi}$ for all 
$t \in [0,+\infty)$. Therefore, we have $\omega\in L^\infty(\R; L^2(\Omega)) 
\cap \cS'(\R_+;L^2(\Omega))$, and since $p \mapsto W(p+1)$ is holomorphic 
in $\mathbb C_+$, we infer from \eqref{sup} with $k=2$, \eqref{bm}, 
Theorem 19.2 and the following remark in \cite{Ru} that  
$\cL[\omega](p)=W(p+1)$ for all $p\in\mathbb C_+$. 
As a consequence, the function
\bel{z2}
w(t):=e^{t} \omega(t) =\frac{1}{2i\pi} \int_{-i\infty}^{i\infty} e^{t (p+1)} 
W(p+1) dp= \frac{1}{2i\pi} \int_{1-i\infty}^{1+i\infty} e^{t p} W(p) dp,
\ t \in \R
\ee
verifies
\bel{lap1}
\cL[w](p)=\cL[\omega](p-1)=W(p),\ p\in\{z\in\mathbb C;\ \re z 
\in (1,+\infty) \}.
\ee
Next \eqref{l1a}--\eqref{l1c} imply 
$$ 
\| W(1+i\eta) \|_{L^2(\Omega)} \leq \rho_0^{-1} \max \left( 2 , 
c_* \left( \frac{\pi}{4} \right) \right) \langle \eta \rangle^{-3 + \alpha_M - 
\alpha_0},\ \eta \in \R \setminus (-1,1),
$$
and we infer from \eqref{sup} with $k=1$ that the mapping 
$\eta \mapsto (1+i\eta) 
W(1+i\eta)\in L^1(\R;L^2(\Omega))$.
Therefore we have
\bel{z3}
y(t):=\pd_t w(t)=\frac{1}{2i\pi} \int_{1-i\infty}^{1+i\infty} e^{t p} pW(p) dp,
\ t \in \R
\ee
by \eqref{z2},  and
\bel{lap2}
\cL[y](p)=p\cL[w](p)=pW(p),\ p\in\{z\in\mathbb C;\ \re z \in (1,+\infty) \}
\ee
from \eqref{lap1}. 
Furthermore, due to \eqref{z3} and the analyticity of the mapping $p \mapsto 
e^{t p} pW(p)$ in $\mathbb C \setminus \R_-$, arising 
from Proposition \ref{l1}, the following identity
\bel{t1b}
y(t)=\frac{1}{2 i \pi} \int_{\gamma(\epsilon,\theta)} e^{t p} pW(p) dp,\ t 
\in \R_+,
\ee
holds for any $\epsilon \in (0,1)$ and $\theta \in \left( \frac{\pi}{2},\pi 
\right)$, where $\gamma(\epsilon,\theta)$ is defined by \eqref{g1}-\eqref{g2}.
Here we used the Cauchy formula and took advantage of the fact that
$$
\lim_{\eta \to +\infty} \int_{\eta ( (\tan \theta)^{-1} \pm i )}
^{1 \pm i \eta} e^{t p} pW(p) d p =0,\ t \in \R_+ 
$$ 
in $L^2(\Omega)$. Indeed, for any sufficiently large $\eta \in (1,+\infty)$ 
and all $t \in \R_+$, \eqref{l1a}--\eqref{l1c} yield
\beas
\left\| \int_{\eta ( (\tan \theta)^{-1} \pm i )}^{1 \pm i \eta} 
e^{t p} pW(p) d p \right\|_{L^2(\Omega)} 
& =  &
\left\| \int_{\eta (\tan \theta)^{-1}}^{1} e^{t (\mu \pm i \eta)} (\mu \pm i 
\eta)W(\mu \pm i \eta) d \mu\right\|_{L^2(\Omega)} \nonumber \\
& \leq & C e^t (1 - \eta (\tan \theta)^{-1}) \eta^{-2+\alpha_M-\alpha_0} 
\| u_0 \|_{L^2(\Omega)}
\eeas
for some positive constant $C$ depending only on $\theta$, $\alpha_0$, 
$\alpha_M$, $\rho_0$ and $\rho_M$.

We turn now to estimating the right-hand side of \eqref{t1b}. First, 
by performing the change of variable $p=\epsilon e^{i \beta}$ with 
$\beta \in (-\theta,\theta)$ in the integral $\int_{\gamma_0(\epsilon,\theta)} 
e^{t p} pW(p) d p$, we derive from \eqref{l1a} and \eqref{a14} that
\bea
\left\| \int_{\gamma_0(\epsilon,\theta)} e^{t p} pW(p) d p 
\right\|_{L^2(\Omega)} 
& \leq & \rho_M \left( \int_{-\theta}^\theta C(\epsilon,\beta) e^{t \epsilon 
\cos \beta} d \beta \right) \epsilon^{-(1+\alpha_M-\alpha_0)} \| u_0 \|
_{L^2(\Omega)} \nonumber \\
& \leq &C e^{t \epsilon} \epsilon^{-(1+2(\alpha_M-\alpha_0))} \| u_0 \|
_{L^2(\Omega)},\ t \in \R_+. \label{a16}
\eea
Next, since \eqref{l1b}-\eqref{l1c} yield the existence of a positive constant 
$C_\theta$ depending only on $\alpha_0$, $\alpha_M$, $\rho_0$, $\rho_M$, and 
$\theta$, such that the estimate $C(r,\theta) \leq C_\theta$ holds uniformly 
in $r \in (0,+\infty)$.  Then it follows from \eqref{l1a} that
\bea
&&\left\| \int_{\gamma_\pm(\epsilon,\theta)} e^{t p} pW(p) d p \right\|
_{L^2(\Omega)} 
\leq \rho_M C_\theta \left( \int_\epsilon^1 r^{-(2+\alpha_M-\alpha_0)} dr 
+ \int_1^{+\infty} r^{-(2-(\alpha_M-\alpha_0))} dr \right) \| u_0 \|
_{L^2(\Omega)} \nonumber \\
& \leq & \frac{\rho_M C_\theta}{1+\alpha_M-\alpha_0} \left( 
\epsilon^{-(1+\alpha_M-\alpha_0)} + \frac{2}{1-(\alpha_M-\alpha_0)} \right) 
\| u_0 \|_{L^2(\Omega)}. \label{a17}
\eea
Now, taking $\epsilon = t^{-1} \in (0,1)$ in \eqref{a16}-\eqref{a17}, we deduce from \eqref{g1} and \eqref{t1b} that
\bel{a18}
\| y(t) \|_{L^2(\Omega)} \leq C t^{1+2(\alpha_M-\alpha_0)} \| u_0 \|
_{L^2(\Omega)},\ t \in (1,+\infty) 
\ee
for some positive constant $C$ depending only on $\theta$, $\alpha_j$ and 
$\rho_j$ for $j=0,M$. Similarly, by choosing $\epsilon=1 \slash 2$ 
in \eqref{a16}-\eqref{a17}, we find that $\| y(t) \|_{L^2(\Omega)} 
\leq C \| u_0 \|_{L^2(\Omega)}$ for all $t \in [0,1]$, 
where $C \in (0,+\infty)$ is independent of $t$. 
Therefore we have $t \mapsto \langle t \rangle^{-3} y(t) \in 
L^\infty(\R;L^2(\Omega))$, and consequently $y\in\cS'(\R_+;L^2(\Omega))$
by \eqref{ww}.
Moreover both functions $p \mapsto \cL[y](p)$ and $p\mapsto pW(p)$ are 
holomorphic in $\mathbb C_+$, and \eqref{lap1} entails  that $\cL[y](p)=pW(p)$ 
for all $p\in\mathbb C_+$, by the unique continuation. 
As a consequence, $v:= \pd_t y \in \cS'(\R_+;L^2(\Omega))$ satisfies 
$\cL[v](p) = p^2 W(p) = (A_q + \rho(x)p^{\alpha(x)})^{-1} \rho(x) 
p^{\alpha(x)-1} u_0$ for every $p \in \C_+$, which shows that $u:=v_{|Q}$ 
is a weak solution to \eqref{eq1} associated with $f=0$. 
Moreover, since $u$ is unique, as can be seen from Definition \ref{d1},
we are left with the task of establishing \eqref{t1a} in the case 
where $f=0$, that is, 
\bel{a19}
u(t) =\frac{1}{2i\pi} \int_{\gamma(\epsilon,\theta)} e^{t p} (A_q + \rho(x)
p^{\alpha(x)})^{-1} \rho(x)p^{\alpha(x)-1} u_0 dp,\ t\in(0,T].
\ee
This equality follows from Proposition \ref{l1} and the identity
$u=\pd_t y$ in $(C_0^\infty)'(0,T;L^2(\Omega))$. 
Indeed, for all $p \in \gamma(\epsilon,\theta)$, the mapping $t \mapsto e^{tp} 
pW(p)$ is continuously differentiable in $(0,T)$, and \eqref{l1a}--\eqref{l1c} 
yield the existence of a constant $C=C(\alpha_0,\alpha_M,\rho_0,\rho_M,
\theta) \in (0,+\infty)$ such that we have
$$ 
\| e^{tp} p^2 W(p) \|_{L^2(\Omega)} \leq C e^{t r \cos \theta} 
\max_{\sigma=\pm 1} r^{-1 + \sigma (\alpha_M-\alpha_0)} \|  u_0 \|
_{L^2(\Omega)},\ p=r e^{\pm i \theta},\ r \in (\epsilon,+\infty).
$$
Moreover, by $\cos \theta \in (-1,0)$, we see that 
$r \mapsto e^{t r \cos \theta} \max_{\sigma=\pm 1} r^{-1 + \sigma (\alpha_
M-\alpha_0)} \in L^1(\epsilon,+\infty)$ for each $t \in (0,T]$, and so 
the integral $\int_{\gamma(\epsilon,\theta)} e^{t p} p^2 W(p) dp$ is 
well-defined. Therefore we obtain \eqref{a19} by this and 
$u \in C((0,T];L^2(\Omega))$.\\

2) We turn now to establishing \eqref{t1a} in the case where $u_0=0$.  
To this purpose, we introduce the following family of operators acting in 
$L^2(\Omega)$,
$$ 
\widetilde{W}(p):=p^{-2}(A_q+\rho(x)p^{\alpha(x)})^{-1} ,\ p \in \mathbb \C 
\setminus \R_-.
$$
For any $\mu \in [1,+\infty)$ and $\eta \in \R$, it follows from \eqref{l1a} 
and \eqref{a14} that $\norm{\widetilde{W}(\mu+i \eta)}_{\cB(L^2(\Omega))}$ 
is majorized by $\langle \eta \rangle^{-2+\alpha_M-2 \alpha_0}$ up to 
some multiplicative constant that is independent of $\eta$ and $\mu$. 
Therefore we have
$$
\sup_{\mu \in [1,+\infty)} \int_\R \norm{\widetilde{W}(\mu+i \eta)}_{\cB(L^2(\Omega))}^k d\eta <\infty,\ k=1,2.
$$
Thus, by arguing exactly in the same way as in the first part of the proof, 
we see that 
\bel{a14b}
S_1(t):=\frac{1}{2 i \pi} \int_{1-i \infty}^{1+i \infty} e^{tp} p 
\widetilde{W}(p) dp,\ t \in \R
\ee
lies in $\cS'(\R_+;\cB(L^2(\Omega)))$, 
\bel{esss1}
t \mapsto \langle t \rangle^{-\alpha_M} S_1(t) 
\in L^\infty(\R;\cB(L^2(\Omega)))
\ee
and 
\bel{lap11}
\cL[S_1\psi](p)=p \widetilde{W}(p)\psi,\ p \in \C_+,\ \psi\in L^2(\Omega).
\ee
By $\tilde{f}$ we denote the extension of a function $f$ by $0$ on 
$(\R \times \Omega) \setminus ((0,T)\times \Omega)$. 
We recall that there exists $\zeta \in \R_+$ such that
\bel{a15}
\langle t \rangle^{-\zeta} \tilde{f} \in L^\infty(\R;L^2(\Omega)),\ 
\ee
and consider the convolution of $S_1$ with $\tilde{f}$, that is, 
$$ 
(S_1*\tilde{f})(t,x) =\int_0^t S_1(t-s) f(s,x) 1_{(0,T)}(s) ds,
\ (t,x) \in \R \times \Omega. 
$$
Evidently, $(S_1*\tilde{f})(t)=0$ for all $t \in \R_-$, and we infer from 
\eqref{esss1} and \eqref{a15} that
\bel{esss2}
\norm{(S_1*\tilde{f})(t)}_{L^2(\Omega)}  \leq \norm{\langle t \rangle
^{-\alpha_M} S_1}_{L^\infty(\R_+;\cB(L^2(\Omega)))} 
\norm{\langle t \rangle^{-\zeta} f}_{L^\infty(\R_+;L^2(\Omega))} 
\langle t \rangle^{1+\alpha_M+\zeta},\ t \in \R_+.
\ee
Therefore $t \mapsto \langle t \rangle^{-1+\alpha_M+\zeta} (S_1 * \tilde{f})(t)
\in L^\infty(\R;L^2(\Omega))$, and consequently 
$S_1*\tilde{f}\in \cS'(\R_+;L^2(\Omega))$.
Moreover, again by \eqref{esss1} and \eqref{a15}, we see that
$$ 
\inf \{\epsilon \in \R_+;\ e^{-\epsilon t} S_1 \in L^1(\R;\cB(L^2(\Omega)))\}
= \inf\{\epsilon \in \R_+;\ e^{-\epsilon t} \tilde{f} \in L^1(\R;L^2(\Omega))\}
=0,
$$
which entails
$$
\cL[S_1*\tilde{f}](p)=\cL[S_1](p) \cL[\tilde{f}](p)=\cL[S_1](p) F(p),\ 
p\in\mathbb C_+,
$$
with $\cL[S_1](p)=\int_0^{+\infty} S_1(t)e^{-pt}dt$ and $\cL[\tilde{f}](p)
= \int_0^{+\infty} \tilde{f}(t)e^{-pt}dt$.
Thus, setting $\tilde{v}:=\pd_t(S_1*\tilde{f})\in \cS'(\R_+;L^2(\Omega))$, 
we derive from \eqref{lap11} that
$$
\cL[\tilde{v}](p) = p\cL[S_1*\tilde{f}](p) = p\cL[S_1](p)F(p)=(A_q + \rho(x)
p^{\alpha(x)})^{-1} F(p),\ p \in \C_+.
$$
It remains to show that
\bel{fin}
\tilde{v}(t)=\int_0^tS_1(t-\tau)\tilde{f}(\tau)d\tau+ S_2\tilde{f}(t),\ 
t \in [0,T].
\ee
This can be done with the aid of \eqref{a14b}, yielding
$$(S_1*\tilde{f})(t)=\frac{1}{2 i \pi} \int_0^t \int_{1-i \infty}^{1+i \infty} 
e^{(t-s)p} p^{-1}(A_q+\rho(x)p^{\alpha(x)})^{-1}\tilde{f}(s) dp ds,\ 
t \in \R_+. $$
Indeed, we notice with a slight adaptation of the reasoning used in the 
derivation of \eqref{t1b} that the integral
$\int_{1-i \infty}^{1+i \infty} e^{(t-s)p} p^{-1}(A_q+\rho(x)p^{\alpha(x)})
^{-1}f(s)dp$ can be replaced on the right-hand side of the above identity by 
$\int_{\gamma(\epsilon,\theta)}e^{(t-s)p} p^{-1}(A_q+\rho(x)p^{\alpha(x)})^{-1}
\tilde{f}(s)dp$ associated with any $\epsilon \in (0,1)$ and 
$\theta \in \left( \frac{\pi}{2},\pi \right)$. 
Therefore we have
$$ 
(S_1 * \tilde{f})(t)=\frac{1}{2 i \pi}\int_0^t\int_{\gamma(\epsilon,\theta)}
e^{(t-s)p} p^{-1}(A_q+\rho(x)p^{\alpha(x)})^{-1}\tilde{f}(s)dp\ ds.
$$
Hence, by \eqref{l1a}, \eqref{a14} and \eqref{a15}, we infer from the Fubini 
theorem that
$$
(S_1*\tilde{f})(t)=\frac{1}{2 i \pi} \int_{\gamma(\epsilon,\theta)} g_q(t,p) 
dp,\ t \in \R_+ 
$$
with
\bel{a15b} 
g_q(t,p):= \int_0^te^{(t-s)p} p^{-1}(A_q+\rho(x)p^{\alpha(x)})^{-1}
\tilde{f}(s)ds,\  p \in \gamma(\epsilon,\theta).
\ee
Therefore, for all $t \in \R_+$ and all $p \in  \gamma(\epsilon,\theta)$, 
we have 
\beas
\pd_t g_q(t,p)
& =& \int_0^te^{(t-s)p} (A_q+\rho(x)p^{\alpha(x)})^{-1}\tilde{f}(s)ds 
+ p^{-1}(A_q+\rho(x)p^{\alpha(x)})^{-1}\tilde{f}(t),
\eeas
and consequently
$$
\norm{\pd_t g_q(t,p)}_{L^2(\Omega)}
\leq \norm{(A_q+\rho(x)p^{\alpha(x)})^{-1}}_{\cB(L^2(\Omega))} 
\left(  \int_0^t e^{s\re p} ds +| p |^{-1} \right) \norm{\tilde{f}}
_{L^\infty(0,t+1;L^2(\Omega))}.
$$
From this and \eqref{l1a}--\eqref{l1c}, it follows that
$$
\norm{\pd_t g_q(t,p)}_{L^2(\Omega)}
\leq  \rho_0^{-1} \max \left( 2, c_*(\theta) \right) 
(1+| \cos \theta |^{-1} )| p |^{-(1+ \alpha_0)} \norm{\tilde{f}}
_{L^\infty(0,t+1;L^2(\Omega))}. 
$$
As a consequence, the mapping $p \mapsto \pd_t g_q(t,p) \in 
L^1(\gamma(\epsilon,\theta);L^2(\Omega))$ for any fixed $t \in \R_+$ 
and $\tilde{v}(t)=\partial_t[S_1*\tilde{f}](t)=\frac{1}{2 i \pi} 
\int_{\gamma(\epsilon,\theta)} \pd_t g_q(t,p) dp$, or equivalently
$$
\tilde{v}(t)=\frac{1}{2 i \pi} \int_{\gamma(\epsilon,\theta)}\left(\int_0^t 
e^{(t-s)p} (A_q+\rho(x)p^{\alpha(x)})^{-1}\tilde{f}(s)ds+p^{-1}(A_q+\rho(x)
p^{\alpha(x)})^{-1}\tilde{f}(t)\right)dp
$$
in virtue of \eqref{a15b}. Now, applying the Fubini theorem to the right-hand 
side of the above identity, we obtain \eqref{fin}. 
This establishes that the restriction to $Q$ of the function expressed 
by the right-hand side of \eqref{fin}, is a weak solution to \eqref{eq1}
associated with $u_0=0$. Evidently such a function lies in 
$\mathcal C([0,T];L^2(\Omega))$. Moreover it is unique from Definition 
\ref{d1}.

Finally, by superposition, the desired result follows readily from 1) and 2).

\subsection{Proof of Remark \ref{rm-cl}}
\label{sec-pr-cl}
We use the notations of Section \ref{sec-prt1}. For $\epsilon \in (0,1)$, 
$\theta \in \left( \frac{\pi}{2}, \pi \right)$ and $R \in [1,+\infty)$, 
we introduce $\gamma_R(\epsilon,\theta):= \{ z \in \gamma(\epsilon,\theta);
\ | z | \in [0,R] \}$ and put 
$\cC_R(\theta) := \{ z \in \C; \ z= R e^{i \beta},\ \beta 
\in [-\theta,\theta] \}$. In light of Proposition \ref{l1}, 
the Cauchy formula yields
$$ 
\int_{\gamma_R(\epsilon,\theta) \cup \cC_R(\theta)^-} p^{-1} (A_q+\rho(x)
p^{\alpha(x)})^{-1} \psi\ dp = 0,\ \psi \in L^2(\Omega), 
$$
where the notation $\cC_R(\theta)^-$ stands for the counterclockwise oriented 
$\cC_R(\theta)$. 
Thus, by letting $R$ to $+\infty$ in the above identity, we obtain
\bel{r1}  
S_2 \psi = \lim_{R \to +\infty} \int_{\cC_R(\theta)} p^{-1} (A_q+\rho(x)
p^{\alpha(x)})^{-1} \psi\ dp,\ \psi \in L^2(\Omega)
\ee
from the definition of $S_2$. Furthermore, for any $R \in [1,+\infty)$, 
from \eqref{l1a} and \eqref{a14}, we have
\bel{r2}
\norm{\int_{\cC_R(\theta)} p^{-1} (A_q+\rho(x)p^{\alpha(x)})^{-1}\ dp}
_{\cB(L^2(\Omega))} \leq C R^{\alpha_M-2 \alpha_0},
\ee
where $C$ is a positive constant which is independent of $R$. 
Since $\alpha_M-2 \alpha_0$ is negative by the assumption, 
we have $S_2 \psi =0$ for any $\psi \in L^2(\Omega)$ directly 
from \eqref{r1}-\eqref{r2}. Finally \eqref{t1a-cl} follows readily from 
this and \eqref{t1a}.


\section{Analysis of the inverse problem}
\label{sec-inv}
In this section, we suppose that $\partial \Omega$ is $\cC^{1,1}$ and 
\eqref{eq-DLaplace} holds, that is, $A_0 = -\Delta$ and 
$$
D(A_0)=H_0^1(\Omega) \cap H^2(\Omega).
$$
We recall for further use that the norm in $H^2(\Omega)$ is equivalent to 
the norm in $D(A_0)$ or in $D(A_q)$.

First we prove that the boundary operator $\mathcal N_{\alpha,\rho,q}(t)$ 
expressed by \eqref{bound}, is well-defined for all $t \in (0,T]$.

\subsection{Definition of the boundary operator}
\label{sec-wp}

By \eqref{bound} and the continuity of the trace operator 
$\varphi \mapsto \pd_\nu \varphi$ from $H^2(\Omega)$ into $L^2(\pd \Omega)$, 
it suffices to prove the following well-posedness for the 
initial-boundary value problem \eqref{eq11}.

\begin{prop}
\label{pr1} 
Let $\alpha$, $\rho$ and $q$ be the same as in Theorem \ref{t1}. 
Then, for all $g\in H^{3\slash2}(\partial\Omega)$, there exists a 
unique weak solution in $\cC([0,+\infty);H^2(\Omega))$ to 
\eqref{eq11}.
\end{prop}
\begin{proof}
Let $G \in H^2(\Omega)$ satisfy $G=g$ on $\partial\Omega$.
Then we notice that $u=u_g$ is a solution to \eqref{eq11} 
if and only if the function $v(t,x):=u(t,x)-t^kG(x)$ is a solution 
to the system
\bel{eq12}
\left\{\begin{array}{rcll} 
(\rho(x)\partial_t^{\alpha(x)} + \cA_q) v(t,x) & = & f(t,x), & 
(t,x)\in  (0,+\infty)\times\Omega,\\
v(t,x) & = & 0, & (t,x) \in (0,+\infty)\times\partial\Omega, \\  
v(0,x) & = & 0, & x \in \Omega,
\end{array}
\right.
\ee
where $f(t,x) := -\left( \rho(x) \partial_t^{\alpha(x)} t^k + t^k \cA_q 
\right) G(x)$.

Furthermore, since  $f \in \mathcal C((0,+\infty);L^2(\Omega))$ and 
$(1+t)^{-k-1}f\in L^\infty(0,+\infty;L^2(\Omega))$, the initial-boundary 
value problem \eqref{eq12} admits a unique weak solution 
$v \in \cC((0,+\infty);L^2(\Omega))$ according to Theorem \ref{t1}. 
Let us now prove that $v \in \cC([0,+\infty);H^2(\Omega))$. For this purpose, 
we infer from the basic identity
$\cL[t^k](p):=\int_0^{+\infty} t^k e^{-pt}dt=\frac{k!}{p^{k+1}}$ that
$$ 
F(p,x):=\cL[f(\cdot,x)](p)=-\frac{k!}{p^{k+1}} \left( \cA_q + \rho(x) 
p^{\alpha(x)} \right) G(x),\ (p,x) \in \C_+ \times \Omega. 
$$
Next, upon extending the expression of the right-hand side of the above 
equality to all $p \in \C \setminus \R_-$, we obtain from the first 
equation of \eqref{eq12} that $V:=\cL[v]$ reads
\bel{It1b}
V(p,x) =-\frac{k!}{p^{k+1}} (A_q+\rho(x)p^{\alpha(x)})^{-1} \left( 
\cA_q + \rho(x) p^{\alpha(x)}\right) G(x), \ (p,x) \in (\C \setminus \R_-) 
\times \Omega.
\ee
Therefore, arguing in the same way as in the proof of Theorem \ref{t1}, 
we obtain for any fixed $\epsilon \in (0,1)$ and $\theta 
\in (\pi \slash 2,\pi)$ that
\bel{It1c}
v(t,x)=\frac{1}{2 i \pi} \int_{\gamma(\epsilon,\theta)} e^{t p} V(p,x) dp,
\ (t,x) \in [0,+\infty) \times \Omega.
\ee
On the other hand, \eqref{It1b} and Lemma \ref{l1} yield that $V(p,\cdot) 
\in D(A_q)$ for all $p \in \C \setminus \R_-$ with
\bel{It1z}
A_qV(p,x) =\frac{k!}{p^{k+1}} \left( \rho(x) p^\alpha(x) 
(A_q+\rho(x) p^{\alpha(x)})^{-1} - I \right) \left( \rho(x) p^{\alpha(x)} 
+ \cA_q \right) G(x),\ x \in \Omega.
\ee
Here the symbol $I$ stands for the identity operator in $L^2(\Omega)$. 
Applying \eqref{l1a}-\eqref{l1b}, we deduce from \eqref{It1z} that
\bea
\norm{A_qV(r e^{\pm i\theta})}_{L^2(\Omega)} & \leq& C r^{-(k+1)} \max(r^{2 
\alpha_M-\alpha_0},r^{2 \alpha_0-\alpha_M} ) \nonumber \\
& \leq& C \epsilon^{-3(\alpha_M-\alpha_0)} r^{2\alpha_M-\alpha_0-k-1} 
\nonumber \\
& \leq & C \epsilon^{-3(\alpha_M-\alpha_0)} r^{1-k-\alpha_0},\ 
r \in [\epsilon,+\infty)                     \label{It1e}
\eea
with some positive constant 
$C=C(\theta,M,\norm{g}_{H^{3 \slash 2}(\pd \Omega)},
\norm{q}_{L^{\infty}(\Omega)},\alpha_0,\alpha_M,\rho_0,\rho_M)$ which is 
independent of $\epsilon$.
Therefore we have $r \mapsto A_q V(r e^{\pm i\theta},\cdot) \in 
L^1(\epsilon,+\infty;L^2(\Omega))$ and hence
$r \mapsto V(re^{\pm i\theta},\cdot)\in L^1(\epsilon,+\infty; D(A_q))$. 
From this and \eqref{It1c}, it follows that 
$v(t,\cdot) \in D(A_q)$ for all $t \in [0,+\infty)$ with
\bel{It1d}
A_qv(t,\cdot)=
\frac{1}{2 i \pi}\int_{\gamma(\epsilon,\theta)} e^{t p} A_qV(p) dp,
\ee
proving that $A_qv \in \cC([0,+\infty);L^2(\Omega))$. As a consequence, 
we have $v \in \cC([0,+\infty);D(A_q))$ and the desired result 
follows immediately from this and the identity $D(A_q)=H_0^1(\Omega) \cap 
H^2(\Omega)$.
\end{proof}

\subsection{Proof of Theorem \ref{It2}}
\label{sec-ip}
The proof of Theorem \ref{It2} is by means of the analytic properties of the 
mapping $t \mapsto \mathcal N_{\alpha,\rho,q}(t)$, defined by \eqref{bound}, 
that are preliminarily established in the coming subsection.

\subsubsection{On the analyticity of the boundary operator}
We first introduce the following notations. Let $X$ be a Hilbert space, and 
let $\mathcal{O}$ be either a subinterval of $\R$ or an open subset of $\C$. 
We denote by $\mathcal A(\mathcal{O};X)$ the space of $X$-valued functions 
that are analytic in $\mathcal O$.

\begin{lem}
\label{ll1}
Let $g \in H^{3\slash2}(\partial\Omega)$ and let $u$ be the solution 
in $\cC([0,+\infty);H^2(\Omega))$ to \eqref{eq11} associated with $g$,
whose existence is guaranteed by Proposition \ref{pr1}. 
Then the mapping $t \mapsto \partial_\nu u(t,\cdot)_{|\partial\Omega}$ lies 
in $\mathcal A((0,+\infty);L^2(\partial\Omega))$.
\end{lem}
\begin{proof} 
By the definitions and the notations used in the proof of 
Proposition \ref{pr1}, the solution $u$ to \eqref{eq11} reads 
$u(t,x)=t^kG(x)+v(t,x)$ for a.e. $(t,x) \in (0,+\infty) \times \Omega$, where 
$v \in \cC([0,+\infty);H^2(\Omega))$ is a solution to \eqref{eq12}. 
Since $G \in H^2(\Omega)$, it is apparent that $t \mapsto t^k\partial_
\nu G_{|\partial\Omega}\in \mathcal A((0,+\infty);L^2(\partial\Omega))$. 
Therefore we are left with the task of showing that $t \mapsto \partial_\nu 
v(t,\cdot)_{|\partial\Omega} \in\cA((0,+\infty);L^2(\partial\Omega))$. 
Since $D(A_q)=H_0^1(\Omega) \cap H^2(\Omega)$ and 
the trace map $w \mapsto \partial_\nu w_{|\partial\Omega}$ is continuous 
from $H^2(\Omega)$ into $L^2(\partial \Omega)$, it is sufficient to 
prove that $t \mapsto v(t,\cdot)\in\mathcal A((0,+\infty);D(A_q))$. 

For this purpose, we fix $\theta_1\in (0,\theta-\pi\slash2) \cap 
(0,\pi-\theta)$, put $\cO :=\{\tau e^{i\psi};\ \tau \in (0,+\infty),\ \psi 
\in(-\theta_1,\theta_1) \}$, and we extend $v$ into a function of 
$\mathcal A(\cO;D(A_q))$. This can be done with the help of 
\eqref{It1e}-\eqref{It1d} by noticing 
$$
|e^{z p}|=|e^{\tau r e^{i(\pm \theta +\psi)}}|
= e^{\tau r \cos( \pm \theta +\psi)} \quad
\mbox{for all $z=\tau e^{i\psi} \in \cO$ 
and $p=re^{\pm i\theta}$ with $r \in[\epsilon,+\infty)$}.
$$

Indeed, since we have $\theta +\psi\in (\theta-\theta_1,\theta+\theta_1)
\subset (\pi\slash2,\pi)$ and 
\\
$-\theta +\psi\in (-\theta-\theta_1,-\theta+\theta_1)\subset 
(-\pi,-\pi\slash2)$, 
it holds true that $\cos( \pm \theta +\psi)\leq \cos(\theta-\theta_1)$ and
\bel{eq-en}
|e^{zp}| \leq e^{|z| r \cos(\theta-\theta_1)},\ z \in \cO,\ p=re^{\pm i\theta},
\ r \in[\epsilon,+\infty).
\ee
Furthermore, since  $\cos(\theta-\theta_1) \in(-1,0)$, 
it follows from \eqref{It1e} and \eqref{eq-en} that
$$
\mathcal W:z\mapsto \frac{1}{2 i \pi}\int_{\gamma(\epsilon,\theta)} e^{z p} 
A_qV(p) dp 
$$
is well defined in $\cO$. Moreover, for any compact subset 
$K \subset \mathcal O$ in $\mathbb C$, we infer from \eqref{It1e} that
$$
\norm{e^{z p} A_qV(r e^{\pm i\theta}) }_{L^2(\Omega)} 
\leq C \epsilon^{-3(\alpha_M-\alpha_0)} e^{\delta r \cos(\theta-\theta_1)} 
r^{1-k-\alpha_0},\ z\in K,\ r \in [\epsilon,+\infty),
$$
where $\delta :=\inf\{|z|;\ z\in K\}>0$ and $C$ is the constant in 
\eqref{It1e}. 
Next, as $z\mapsto e^{z p} A_qV(p) \in \cA(\cO;L^2(\Omega))$ for all 
$p \in \{ r e^{\pm i\theta};\ r \in [\epsilon,+\infty) \}$, this entails that
$\mathcal W \in \mathcal A(\mathcal O;L^2(\Omega))$. Furthermore, since 
$\mathcal W(t)=A_q v(t,\cdot)$ for all $t \in (0,+\infty)$, 
we obtain by \eqref{It1e} that 
\bel{eq-an}
t\mapsto A_qv(t,\cdot) \in\mathcal A((0,+\infty);L^2(\Omega)).
\ee
Finally, arguing in the same way as above, we deduce from \eqref{It1c} that 
$t\mapsto v(t,\cdot)\in\mathcal A((0,+\infty);L^2(\Omega))$. 
This and \eqref{eq-an} yield that $t \mapsto v(t,\cdot)\in\mathcal A
((0,+\infty);D(A_q))$, which proves the result.
\end{proof}

In terms of Lemma \ref{ll1}, we can complete the proof of Theorem \ref{It2}.

\subsubsection{Completion of the proof}
For $j=1,2$, we denote by $u_j$ the weak solution to the initial-boundary 
value problem \eqref{eq11} associated with $g\in \Hin$, 
$(\alpha,\rho,q)=(\alpha_j,\rho_j,q_j)$, and $T=+\infty$. The proof is divided 
into three steps. The first one is to establish that
\bel{It2d}
\partial_\nu u_1(t,\cdot)_{|\Go}=\partial_\nu u_2(t,\cdot)_{|\Go},\ 
t \in(0,+\infty),
\ee
and the second one is to derive from \eqref{It2d} that the functions 
$U_j :=\mathcal L[u_j]$, $j=1,2$, verify
\bel{It2dd}
\partial_\nu U_1(p,\cdot)_{|\Go}=\partial_\nu U_2(p,\cdot)_{|\Go},\quad 
p\in(0,+\infty).
\ee
The third step corresponds to the end of the proof, which is by means of 
the existing results for the Calder\'on problem with partial Cauchy data.

\noindent {\bf Step 1.} Put $h(t,x) := \partial_\nu u_1(t,x)-\partial_\nu u_2(t,x)$ for $(t,x) \in 
(0,+\infty) \times \Go$. We recall from Lemma \ref{ll1}  that 
$h \in \mathcal A((0,+\infty);L^2(\Go))$, and from \eqref{It2a} that
$$
h(t_n)=0,\ n \in \N.
$$
Therefore, by \eqref{cond1}, the set of the zeros of the analytic function $h$
has accumulation point in $(0,+\infty)$, so that identically vanishes, 
and \eqref{It2d} follows.

\noindent {\bf Step 2.} For $j=1,2$, let $v_j$ denote the solution to \eqref{eq12} where 
$(\alpha_j,\rho_j,q_j)$ is substituted into $(\alpha,\rho,q)$ such that 
\bel{st2.1}
u_j(t,x)=t^kG(x)+v_j(t,x),\ (t,x) \in Q.
\ee
Furthermore, putting $V_j:=\cL v_j$, we deduce from \eqref{It1b} and 
\eqref{It1e} that
\bea
\left\| \int_{\gamma_\pm(\epsilon,\theta)} e^{t p} A_{q_j} V_j(p) d p \right\|
_{L^2(\Omega)} 
& \leq & C \epsilon^{-3(\alpha_M-\alpha_0)} \int_\epsilon^{+\infty} 
r^{1-k-\alpha_0} dr \nonumber \\
& \leq & \frac{C}{k+\alpha_0-2} \epsilon^{2-k-\alpha_0-3(\alpha_M-\alpha_0)}, 
\label{st2.2}
\eea
where the constant $C$ is the same as in \eqref{It1e}. 
Similarly, by Lemma \ref{l1}, we infer from \eqref{a14} and \eqref{It1z} that
\bea
\left\| \int_{\gamma_0(\epsilon,\theta)} e^{t p} A_{q_j} V_j(p)
d p \right\|_{L^2(\Omega)} & \leq& C \epsilon^{-(k+\alpha_M)} 
\left( \int_{-\theta}^\theta e^{t \epsilon \cos \beta} C_\beta d \beta \right) 
\nonumber \\
& \leq& Ce^{t \epsilon} \epsilon^{-(k+2 \alpha_M-\alpha_0)}, \label{st2.3}
\eea
where another constant $C>0$ is independent of $\epsilon$. Thus, for all $t 
\in (1,+\infty)$, by taking $\epsilon=t^{-1}$ in \eqref{st2.2}-\eqref{st2.3}
we see that
$\norm{\int_{\gamma(\epsilon,\theta)} e^{t p} A_{q_j} V_j(p) d p}
_{L^2(\Omega)}$ is upper bounded by $t^{k+2 \alpha_0 - \alpha_M}$ 
up to some positive constant $C_j$, which is independent of $t$. 
In light of \eqref{It1d}, this entails that $\norm{v_j(t,\cdot)}_{H^2(\Omega)} 
\leq C_j t^{k+2 \alpha_0 - \alpha_M}$  for every $t \in (1,+\infty)$. 
Therefore, by \eqref{st2.1} we have
$$
\norm{u_j(t,\cdot)}_{H^2(\Omega)}\leq C_j t^{k+2 \alpha_0},\ t \in (1,+\infty). $$
Moreover, since $v_j \in L^\infty(0,1;H^2(\Omega))$ in virtue of 
Lemma \ref{ll1}, and hence $u_j \in L^\infty(0,1;H^2(\Omega))$ 
by \eqref{st2.1}, we obtain that
$t \mapsto e^{-pt}u_j(t,\cdot)\in L^1(0,+\infty;H^2(\Omega))$ for all 
$p \in \C_+$. This and the continuity of the trace map 
$v \mapsto\partial_\nu v_{|\partial\Omega}$ from $H^2(\Omega)$ into $L^2(\partial\Omega)$, yield that
$$
\mathcal L[\partial_\nu u_j](p)=\partial_\nu U_j(p),\ j=1,2,\ p \in \C_+.
$$
Now \eqref{It2dd} follows from this and \eqref{It2d}.

\noindent {\bf Step 3.} We can complete the proof by \cite[Theorem 7]{IY2} (see also \cite{IY1}) 
when $d=2$ and \cite[Theorem 1.2]{KSU} when $d \geq 3$.

\begin{thm}
\label{thm-c}
Assume that $\partial \Omega$ is smooth and that $\Omega$ is connected.
For 
$$
V \in \mathcal V:= \{ q \in L^\infty(\Omega);\ 0\ \mbox{lies in the resolvent 
set of}\ A_q \}, 
$$
let $\Lambda_V$ be the partial Dirichlet-to-Neumann map
$\Hin \ni \phi\mapsto\partial_\nu w_{|\Go}$, 
where $w$ is the solution to 
\bel{ell-bvp}
\left\{\begin{array}{rcll} 
-\Delta w+V(x)w & = & 0, & x \in   \Omega,\\
w(x) & = & \phi(x), & x \in  \pd \Omega. \\  
\end{array}
\right.
\ee
For $j=1,2$, pick $V_j$ in $\mathcal V \cap W^{1,r}(\Omega)$ with 
$r \in (2,+\infty)$, if $d=2$, and in $\mathcal V$ if $d \geq 3$. 
Then
\bel{It2e}
\mbox{$\Lambda_{V_1}=\Lambda_{V_2}$ yields $V_1=V_2$.}
\ee
\end{thm}

It is clear for all $p \in (0,+\infty)$ that
$\tilde{U}_j(p):=\frac{p^{k+1}}{k!} U_j(p)$, $j=1,2$, is a solution to 
\eqref{ell-bvp} associated with
$V=q_j+\rho_jp^{\alpha_j}$ and $\phi=g$. As a consequence, we have
$$
\Lambda_{q_1+\rho_1p^{\alpha_1}}g=\Lambda_{q_2+\rho_2p^{\alpha_2}}g,\quad 
p \in (0,+\infty)
$$
by \eqref{It2dd}, and since $g$ is arbitrary in $\Hin$, this immediately 
entails that 
\bel{It2f}
\Lambda_{q_1+\rho_1p^{\alpha_1}}=\Lambda_{q_2+\rho_2p^{\alpha_2}},\quad 
p \in (0,+\infty).
\ee
Moreover, from the definition of $\mathcal E_d$, 
for every $p \in (0,+\infty)$, we have $q_j+\rho_jp^{\alpha_j} \in W^{1,r}
(\Omega)$ with $r \in (2,+\infty)$ if $d=2$, 
and $q_j+\rho_jp^{\alpha_j} \in L^\infty(\Omega)$ if $d \geq 3$.
Therefore, applying \eqref{It2e} with $V_j=q_j+\rho_jp^{\alpha_j}$, we infer 
from \eqref{It2f} that
\bel{eq-f}
q_1+\rho_1p^{\alpha_1}=q_2+\rho_2p^{\alpha_2},\ p \in (0,+\infty).
\ee
Letting $p$ to zero in \eqref{eq-f}, we see that $q_1=q_2$. 
Thus, taking $p=1$ in \eqref{eq-f}, we obtain that $\rho_1=\rho_2$. 
Finally, applying \eqref{eq-f} with $p=e$, we find that
$e^{\alpha_1}=e^{\alpha_2}$, which yields that $\alpha_1=\alpha_2$. 

\noindent {\bf Acknowledgements.} 
The two first authors would like to thank the Department of Mathematical 
Sciences of The University of Tokyo,
where part of this article was written, for its kind hospitality. All the authors are partially supported by Grants-in-Aid for Scientific 
Research (S) 15H05740 and (S) 26220702, Japan Society for the
Promotion of Science.

\end{document}